\documentclass[12pt]{amsart}
\usepackage{amsmath,amssymb,amsthm,a4wide,xspace,amscd,eucal}
\usepackage[all]{xy}
\numberwithin{equation}{section}
\newtheorem{example}[equation]{Example}
\newtheorem{definition}[equation]{Definition}
\newtheorem{remark}[equation]{Remark}
\newtheorem{theorem}[equation]{Theorem}
\newtheorem{lemma}[equation]{Lemma}
\newtheorem{proposition}[equation]{Proposition}
\newtheorem{corollary}[equation]{Corollary}

\newcommand{\C}{\ensuremath{\mathbb{C}}\xspace}
\renewcommand{\S}{{\mathcal S}}

\renewcommand{\l}{\ensuremath{\lambda}}

\newcommand{\m}{\ensuremath{\mathfrak{m}}}
\newcommand{\p}{\ensuremath{\mathfrak{p}}}

\newcommand{\n}{\ensuremath{\mathfrak{n}}}
\newcommand{\Z}{\ensuremath{\mathbb{Z}}\xspace}

\renewcommand{\phi}{\varphi}
\def\D{\Delta}
\def\m{{\mathfrak m}} \def \Max {\mathfrak{max}} \def
\supp {\mathsf{supp}}

\def\O{{\mathcal O}}
\def\p{{\mathfrak p}}

\newcommand{\g}{\ensuremath{\dot {\mathfrak{g}}}}
\newcommand{\h}{\ensuremath{\dot{\mathfrak{h}}}}
\newcommand{\G}{\ensuremath{\mathfrak{g}}}
\newcommand{\ZZ}{\ensuremath{\mathfrak{z}}}
\newcommand{\MM}{\ensuremath{\mathsf{M}}}
\renewcommand{\H}{\ensuremath{\mathfrak{h}}}

\newcommand{\UU}{\ensuremath{\mathsf{U}}}
\newcommand{\SF}{\ensuremath{\mathsf{S}}}
\newcommand{\SR}{\ensuremath{\mathsf{R}}}
\newcommand{\Bo}{\ensuremath{\mathfrak {b}}}
\newcommand{\VV}{\mathsf{V}}
\newcommand{\LL}{\mathsf{L}}
\newcommand{\QQ}{\mathsf{Q}}
\newcommand{\HH}{\mathsf{H}}
\newcommand{\WW}{\mathsf{W}}
\newcommand{\FF}{\mathsf{F}}
\newcommand{\AS}{\mathsf{A}}
\newcommand{\DD}{\mathsf{D}}
\newcommand{\II}{\tt {I}}
\newcommand{\JJ}{\tt {J}}
\newcommand{\KS}{\mathsf K}
\newcommand{\KK}{\tt {K}}
\newcommand{\tN}{\tt {J}}
\newcommand{\PP}{\mathsf{P}}
\newcommand{\NN}{\mathsf{N}}

\begin{document}
\title{New Irreducible Modules for \break
Heisenberg and Affine Lie Algebras}
\author{Viktor Bekkert} \address{\noindent
Departamento de Matem\'atica, ICEx, Universidade Federal de Minas
Gerais, Av.  Ant\^onio Carlos, 6627, CP 702, CEP 30123-970\\ Belo
Horizonte-MG, Brasil} \email{\,bekkert@mat.ufmg.br}
\author{Georgia Benkart}
\address{\noindent Department of Mathematics, University of
Wisconsin-Madison, Madison, WI 53706, USA} \email{\,benkart@math.wisc.edu}
\author{Vyacheslav Futorny}
\author{Iryna Kashuba}
\address{ Institute of Mathematics, University of
S\~ao Paulo, Caixa Postal 66281 CEP 05314-970, S\~ao Paulo,
Brasil}\email{futorny@ime.usp.br} \email{kashuba@ime.usp.br}

\date{}

\begin{abstract}
We study $\mathbb Z$-graded modules of nonzero level with
arbitrary weight multiplicities over Heisenberg Lie algebras and
the associated generalized loop modules over affine Kac-Moody
Lie algebras. We construct new families of such irreducible modules
over Heisenberg Lie algebras. Our main result establishes the
irreducibility of the corresponding generalized loop modules providing
an explicit construction of many new examples of irreducible modules
for affine Lie algebras. In particular,  to any function
$\phi:\mathbb N\rightarrow \{\pm\}$ we associate a $\phi$-highest
weight module over the Heisenberg Lie algebra and a $\phi$-imaginary
Verma module over the affine Lie algebra. We show that any
$\phi$-imaginary Verma module of nonzero level is irreducible.
\end{abstract}

\maketitle

\section{Introduction}\label{s1}
Affine Lie algebras are the most studied among the
infinite-dimensional Kac-Moody Lie algebras and have widespread
applications. Their representation theory  is far richer than
that of finite-dimensional simple Lie algebras.  In particular, affine
Lie algebras have irreducible modules containing both finite- and
infinite-dimensional weight spaces, something that cannot happen
in the finite-dimensional setting. These representations arise
from taking non-standard partitions of the root system;  that is, partitions
which are not equivalent under the Weyl group to the standard
partition into positive and negative roots (see \cite{DFG}).  For affine Lie algebras,
there are always only finitely many equivalence classes of such
nonstandard partitions (see \cite{F4}).  Corresponding to each partition
is a Borel subalgebra,  and one can form representations induced
from one-dimensional modules for these Borel subalgebras. These
modules, often referred to as \emph{Verma-type modules}, were first studied
by Jakobsen and Kac \cite{JK}, and by Futorny \cite{F3,
F4}. Results on the structure of Verma-type modules can also be found in (\cite{Co, F1,
FS}).

Let $\G$ be an affine Lie algebra, $\H$ its standard Cartan subalgebra,  and $\ZZ = \C c$ its center,
where $c$ is the canonical central element.
Let $\VV$ be a weight $\G$-module, that is,  $\VV=\bigoplus_{\mu\in \H^*}
\VV_{\mu}$, where $\VV_{\mu}=\{v\in \VV\mid hv=\mu(h)v \ \hbox{\rm for all} \  h\in \H\}$. If $\VV$
is irreducible, then $c$ acts as a scalar on $\VV$ called the {\em level}
of $\VV$.    The theory of Verma-type modules is  best developed in
the case when the level is nonzero
\cite{F4}.    For example, the  so-called imaginary Verma
modules  induced from  the natural Borel subalgebra
 are always irreducible when the level is nonzero (\cite{JK}, \cite{F2}).

The classification of irreducible modules is
known only for modules with finite-dimensional weight spaces
(see \cite{FT}) and for certain subcategories of induced modules with
some infinite-dimensional weight spaces (see for example,  \cite{F3}, \cite{FKM},
\cite{FK}).   Our main goal is to go beyond the modules with
finite-dimensional weight spaces and to construct new irreducible
modules of nonzero level with infinite-dimensional weight
spaces. Examples of such modules have been  constructed  previously
by Chari and Pressley in \cite{CP} as the
tensor product of highest and lowest weight modules.

Here we consider different Borel-type subalgebras that do not correspond to
partitions of the root system of $\G$.   Such a subalgebra
is determined by a function $\phi: \mathbb N\rightarrow
\{\pm\}$ on the set $\mathbb N$ of positive integers,  and so is denoted $\Bo_{\phi}$.
The subalgebra $\Bo_{\phi}$ gives rise to a class of $\G$-modules called   $\phi$-{\it imaginary Verma modules}.
These modules can be viewed as induced from $\phi$-highest weight
modules over the Heisenberg subalgebra of $\G$. This construction
is similar to the construction of imaginary Whittaker modules in
\cite{Ch},  but unlike the modules in \cite{Ch},  our modules over the Heisenberg subalgebra are
$\mathbb Z$-graded.  If $\phi(n)=+$ for all $n \in \mathbb N$, then
$\Bo_{\phi}$ is the natural Borel subalgebra of $\G$.

We establish a criterion for the irreducibility of $\phi$-imaginary
Verma modules. It comes as no surprise that any such module is
irreducible if and only if it has a nonzero level.

Next we consider the classification problem for irreducible
$\mathbb Z$-graded modules for the Heisenberg subalgebra of $\G$.
The ones of level zero were  determined  by Chari \cite{C}. Any
such module of  nonzero level with a $\Z$-grading has all its
graded components infinite-dimensional  by \cite{F1};  otherwise,
it is a highest weight module. We classify all {\it admissible
diagonal} $\mathbb Z$-graded irreducible modules of nonzero
level for an arbitrary infinite-dimensional Heisenberg Lie
algebra. Since the $\mathbb Z$-graded components of a module are not assumed to be
finite-dimensional the restriction on a module to be diagonal is natural.
 We show that these modules have a $\mathbb
Z^{\infty}$-gradation and can be obtained from weight modules over
an associated Weyl algebra as in \cite{BBF}  by compression of the
gradation.

The restriction on a module to be admissible
 leads to an equivalence between the category of
admissible diagonal $\mathbb Z$-graded modules for
Heisenberg Lie algebras and the  category of admissible weight
modules for the Weyl algebra $\mathsf{A}_{\infty}$ \cite{BBF}. Examples of
such modules were considered earlier by Casati \cite{Ca}, where
they were constructed by means of an action of differential
operators on a space of polynomials in infinitely many variables
(compare Theorem~\ref{thm-realiz-locally-finite} below). Examples
of non-admissible diagonal $\mathbb Z$-graded irreducible
modules were constructed in \cite{MZ}.

We use parabolic induction to construct
generalized loop modules  for the affine Lie algebra $\G$. The modules are
induced from an arbitrary irreducible $\mathbb Z$-graded  module
of nonzero level for the Heisenberg subalgebra of $\G$. This construction
extends Chari's construction to the nonzero level case. Our
main result establishes the irreducibility of any generalized loop
module  induced from a diagonal irreducible
module of nonzero level for the Heisenberg subalgebra. By this
process,  we obtain
new families of irreducible modules of nonzero level for any affine
Lie algebra.  The irreducible modules constructed in
\cite{Ca} are ``dense''  in the sense that they have the maximal possible set of
weights and hence are different from the ones studied here.

{\em It should be noted that all results in
our paper hold for both the  untwisted and twisted affine
Lie algebras.}

 The structure of the paper is as follows.  In Section~3,  we
 construct $\phi$-imaginary Verma modules for affine Lie algebras
  for any function $\phi:\mathbb N\rightarrow
 \{\pm\}$,
 and establish a criterion for their irreducibility in Theorem \ref{thm-phi-imaginary-irred}. In
 Section~4, we consider various types of modules 
 (torsion, locally-finite, diagonal, and admissible)  for the
 Heisenberg algebra. 
 Theorems~\ref{thm-Z-infty-gradation} and \ref{theorem-irred-gr} (see Corollary
 \ref{cor-irred-gr-G})
 provide the classification
 of all irreducible $\Z$-graded admissible diagonal modules of nonzero level for
 the Heisenberg algebra. Finally, in Section~5,  we introduce
 generalized loop modules for affine Lie algebras and study their structure. These
 modules are induced from finitely generated $\Z$-graded irreducible diagonal modules over the
 Heisenberg subalgebra of $\G$. Our main result is Theorem \ref{theor-irred-gener-loop} which
 establishes the irreducibility of any generalized loop module
  induced from such an irreducible diagonal module of nonzero level for the Heisenberg subalgebra.
Hence, we develop a method for constructing new irreducible modules for
affine Lie algebras starting from an irreducible diagonal
module over the Heisenberg subalgebra.  In general, these modules cannot be obtained
by the pseudo-parabolic induction method
considered in \cite{FK}.

\section{Preliminaries}

Let $\G$ denote an affine Lie algebra over the complex numbers $\mathbb C$.
Associated to $\G$ is a finite-dimensional simple Lie subalgebra $\g$  with Cartan subalgebra
$\h$ and root system  $\dot  \D$.   There are elements $c,d$ of $\G$ (the canonical central element
$c$ and the degree derivation $d$) so that  $\H = \h \oplus \mathbb Cc
\oplus \mathbb Cd$ is a Cartan subalgebra of $\G$, and
$\ZZ = \mathbb Cc$ is the center
of $\G$.    The algebra  $\G$ has a root space decomposition  $$
\G = \H \oplus  \bigoplus_{\alpha \in \H^* \backslash \{0\}}
\G_\alpha
$$ relative to $\H$,
where $\G_\alpha = \{ x \in \G \mid  [h, x] = \alpha(h) x \,
{\text{ for all }} \, h \in \H\}$.    The set  $\Delta = \{ \alpha
\in \H^* \backslash \{0\} \, | \, \G_\alpha \neq 0\}$ is the root
system of $\G$.     Let $\dot \D = \dot \D_+ \cup \dot \D_-$ be
a decomposition of the corresponding finite root system $\dot \D$ of $\g$ into positive and negative roots relative to a base $\Pi$ of simple roots.
When there are two root lengths,  let $\dot \D_l$ and $\dot \D_s$ denote the long and short roots in $\dot \D$ respectively.
The root system $\D$ of $\G$ has a natural partition into positive
and negative roots, $\D = \D_+ \cup \D_-$,  where  $\D_- = -\D_+$.    Moreover,
 $\Delta_+ = \D_+^{\mathsf {re}} \cup \D_+^{\mathsf {im}}$, where the imaginary positive roots
 $\D_+^{\mathsf {im}} =  \{n \delta \mid  n \in \Z_{>0}\}$  are positive integer multiples
 of the indivisible imaginary root $\delta$, and the real positive roots  $\D_+^{\mathsf {re}}$  are given by
 \begin{equation}\label{eq:posreal}
\D^{\mathsf{re}}_+ = \begin{cases}   \{ \alpha + n\delta\ |\ \alpha \in \dot \D_+, \, n \in \Z\},
 \hspace{1.7 truein}  \hbox{\rm  if $r = 1$  (the untwisted case),  }  \\
 \{ \alpha + n\delta\ |\ \alpha \in (\dot \D_s)_+, n \in \Z\}\, \cup \,  \{ \alpha + nr\delta\ |\ \alpha \in (\dot \D_l)_+, n \in \Z \}
  \qquad \hbox{\rm  if $r = 2,3$ and} \\
  \hspace{4.7 truein} \hbox{\rm  not $\mathsf{A}_{2\ell}^{(2)}$ type,}  \\
 \{ \alpha + n\delta\ |\ \alpha \in (\dot \D_s)_+, n \in \Z\}\, \cup \,  \{ \alpha + 2n\delta\ |\ \alpha \in (\dot \D_l)_+, n \in \Z \} \\
 \hspace{1.2 truein} \cup \, \{\frac{1}{2}\left ( \alpha + (2n-1)\delta\right) \ |\ \alpha \in (\dot \D_l)_+, n \in \Z \} \qquad  \qquad \hbox{\rm  if $\mathsf{A}_{2\ell}^{(2)}$ type.}  \\
\end{cases}
\end{equation}

We refer to \cite{K} for basic results on Kac-Moody
theory and for the notation used in \eqref{eq:posreal}.

A subset $\SF$ of $\D$ affords a \emph{partition} of $\D$ if $\SF \cup (-\SF) = \D$
and $\SF \cap (-\SF) = \emptyset$.
A  partition $\D
= \SF \cup ( -\SF)$ is said to be {\it closed} if whenever $\alpha$
 and $\beta$
are in $\SF$ and $\alpha + \beta \in \D$, then $\alpha + \beta
\in\SF$.   For any  $\SF$ giving a closed partition of $\D$,  the spaces $\G_{\SF} =
\bigoplus_{\alpha \in \SF}\G_{\alpha}$ and  $\G_{-\SF} =
\bigoplus_{\alpha \in -\SF}\G_{\alpha}$ are subalgebras of $\G$, and
 $\G = \G_{-\SF} \oplus \H \oplus
\G_\SF$ is a triangular decomposition of $\G$.


A \emph{weight module} $\VV$
with respect to $\H$ has a decomposition $\VV = \bigoplus_{\l \in \H^*} \VV_\l$, where
$\VV_\l = \{v \in \VV \mid hv = \l(h)v$ for all $h \in \H\}$,
and we say that the set of weights of $\VV$  is the \emph{support} of $\VV$
and write  $\supp(\VV)=\{\lambda\in \H^* \mid  \VV_{\lambda}\neq 0\}$.

\subsection{Imaginary Verma modules} \quad

Let $\D = \SF \cup (-\SF)$ denote a closed partition of $\D$.
By the Poincar\'e-Birkhoff-Witt theorem, the triangular
decomposition $\G = \G_{-\SF} \oplus \H \oplus \G_{\SF}$
of $\G$ afforded by $\SF$ determines a triangular
decomposition of the universal enveloping algebra  $\UU(\G)$ of
$\G$ given by $\UU(\G) = \UU(\G_{-\SF}) \otimes \UU(\H)
\otimes \UU(\G_{\SF})$.   Let  $\Bo_{\SF}=\H\oplus \G_{\SF}$  be
the associated Borel subalgebra.   Any $\lambda \in \H^*$
extends to an algebra homomorphism (also denoted by $\lambda$) on the enveloping algebras
$\UU(\H)$ and  $\UU(\Bo_{\SF})$ with zero values on $\G_{\SF}$.
Corresponding to any such $\lambda$ is a one-dimensional $\UU(\Bo_{\SF})$-module $\mathbb Cv$
with  $xv = \l(x)v$ for all $x \in \UU(\Bo_{\SF})$.    The induced module
$$
\MM_{\SF}(\l) = \UU(\G) \otimes_{\UU(\Bo_{\SF})} \mathbb C v,
$$
is a \emph{Verma type module} as defined in \cite{Co} and \cite{FS}.
The canonical central element $c$ acts by multiplication by  the scalar
$\l(c)$ on $\MM_{\SF}(\l)$, and we say that $\l(c)$ is the \emph{level} of $\MM_{\SF}(\l)$.
Clearly, $\MM_{\SF}(\l)\simeq \UU(\G_{-\SF})$ as a
$\G_{-\SF}$-module.

When $\SF = \Delta_+$, the module $\MM_{\SF}(\l)$ is an \emph{imaginary Verma module}. It was shown in
\cite{F2} that  the imaginary Verma module $\MM_{\SF}(\l)$ is
irreducible if and only if $\l(c) \neq 0$.

\section{Verma modules corresponding to the map $\phi$}

\subsection{$\phi$-Verma modules for the Heisenberg subalgebra} \quad

The subspace  $\LL:= \mathbb Cc \oplus \bigoplus_{n \in \Z \setminus \{0\}} \G_{n\delta}$ forms a Heisenberg Lie subalgebra of
the affine algebra $\G$.   Thus,
$[x,y] = \xi(x,y)c$ for all $x \in \G_{m\delta}, y \in \G_{n \delta}$,  where
$\xi(x,y)$ is a certain skew-symmetric bilinear form with $\xi(\G_{m\delta}, \G_{n \delta}) =
0$ if $n \neq - m$, and whose restriction to $\G_{m \delta} \times \G_{-m \delta}$
is nondegenerate for all $m \neq 0$.
The algebra $\LL$ has  a triangular decomposition $\LL = \LL^- \oplus \C c
\oplus \LL^+$, where $\LL^{\pm} = \bigoplus_{n \in \mathbb N} \G_{\pm n \delta}$.

Now let $\phi: \mathbb N\rightarrow \{\pm\}$ be an arbitrary function defined on $\mathbb N = \{1,2,\dots\}$.
 The spaces $$\LL_{\phi}^{\pm}= \Big(\bigoplus_{n\in \mathbb N, \phi(n)=\pm }
\G_{n\delta}\Big)\oplus \Big(\bigoplus_{m\in \mathbb N, \phi(m)=\mp }
\G_{-m\delta}\Big)$$ are abelian subalgebras of $\LL$, and $$\LL= \LL_{\phi}^{-}\oplus \mathbb C c\oplus
\LL_{\phi}^{+}$$ is a triangular decomposition.   Of course, if $\phi(n) = +$ for all $n \in \mathbb N$,
then $\LL_{\phi}^{+} = \LL^+$,  and this is just the triangular decomposition above.

Let $\mathbb C v$ be a one-dimensional
representation of $\mathbb C c\oplus \LL_{\phi}^{+}$, where $cv=a
v$ for some $a\in \mathbb C$ and $\LL_{\phi}^{+}v=0$.  The corresponding \emph{$\phi$-Verma module}  is
the induced module
$$\MM_{\phi}(a)=\UU(\LL)\otimes_{\UU(\mathbb C c\oplus \LL_{\phi}^{+})}\mathbb C
v.$$ Clearly, $\MM_{\phi}(a)$ is free as a $\UU(\LL_{\phi}^{-})$-module of rank 1 generated
by the vector $1 \otimes v$.

When $\phi(n)=+$ for all $n \in \mathbb N$,  then $\MM_{\phi}(a)$ is just the usual
Verma module for the Heisenberg Lie algebra $\LL$. Note that if $\phi_1\neq \phi_2$ then $\MM_{\phi_1}(a)$ and $\MM_{\phi_2}(a)$ are not isomorphic.

\begin{remark}  {\rm Let $\SF = \D_+$ and consider the imaginary Verma module
$\MM_{\SF}(\l)$ where $\l \in \H^*$.      This module has
both finite- and infinite-dimensional weight spaces relative to $\H$.
By \cite{F2}, the sum of the finite-dimensional weight spaces in $\MM_{\SF}(\l)$  is the
Verma module $\MM_{\phi}(\l(c))$  for the Heisenberg
subalgebra $\LL$, where $\phi$ is the function with
$\phi(n) = +$ for all $n \in \mathbb N$.} \end{remark}

Since $\UU(\LL)$ has a natural $\mathbb Z$-gradation, we obtain for an arbitrary function $\phi$ the following:

\begin{proposition}\label{prop:phi-graded}
$\MM_{\phi}(a)$ is a $\mathbb Z$-graded $\LL$-module, where
$$\MM_{\phi}(a)=\bigoplus_{n\in \mathbb Z}\MM_{\phi}(a)_{n},$$
and $\MM_{\phi}(a)_{n}=\UU(\LL_{\phi}^{-})_n v$.   If   $\phi(k)\neq
\phi(\ell)$ for some $k,\ell \in \mathbb N$,  then $\MM_{\phi}(a)_{n}$ is
infinite-dimensional for any $n \in \mathbb Z$.
\end{proposition}

\begin{proof} Suppose $n \in \mathbb Z_{\geq 0}$ and set $\MM = \MM_\phi(a)$.
If $\phi(1) = -$, then there is some $r \in \mathbb N$ so that $\phi(r) = +$.   Let $x \in \LL_\delta$ and $y \in \LL_{-r \delta}$ be nonzero.
Since the vectors  $x^{n+kr}y^{k} v \in \MM_n$ are linearly independent  for all $k \geq 0$,  we have that $\MM_n$ is infinite dimensional.
Similarly,  the vectors $x^{(rk-1)n} y^{kn} v \in \MM_{-n}$ are linearly independent for all $k \geq 1$, so that $\MM_{-n}$ is infinite dimensional as well.
The argument when $\phi(1) = +$ is analogous.
\end{proof}

\begin{proposition}\label{prop-irred-phi-Verma}
$\MM_{\phi}(a)$ is irreducible if and only if $a\neq 0$.
\end{proposition}

\begin{proof}  Let $\{x_i\}_{i \in \mathbb N}$ be a basis of $\LL_\phi^+$ of root vectors, and let
$\{y_i\}_{i \in \mathbb N}$ be the dual basis of $\LL_\phi^-$ so that  $[x_i, y_j] = \delta_{i,j}c$ for all $i,j$.
Set $\bar k = (k_1, k_2, \ldots )$ where $k_i \in \mathbb N \cup \{0\}$ for each $i$ and only finitely
many $k_i$ are nonzero, and say $\bar k < \bar \ell$ if there is some $s$ such that $k_i = \ell_i$ for $i < s$
and $k_s < \ell_s$.    The elements
$x(\bar k) = \prod_{i} x_i^{k_i}$ and $y(\bar k) = \prod_{i}y_i^{k_i}$ are well defined since $\LL_\phi^{+}$
and $\LL_\phi^-$ are abelian.     Moreover, the  vectors $y(\bar k)v$ as $\bar k$ ranges over all such tuples in
$(\mathbb N \cup \{0\})^\infty$ form a basis for $\MM_\phi(a)$.

Now suppose $a \neq 0$, and let $w = \sum_{\bar k} \xi(\bar k) y(\bar k)v$ be a nonzero element of $\MM_\phi(a)$,
where only finitely many of the scalars $\xi(\bar k) \in \mathbb C$ are nonzero.     Let $\bar m$ be the largest
tuple with $\xi(\bar m) \neq 0$.       Then  since $x_i y_i^\ell v = [x_i, y_i^\ell]v = \ell a y_i^{\ell-1}v$ for each $\ell \geq 0$, it follows that

$$x(\bar m)w =  \xi(\bar m) \left(\prod_{i} m_i ! \right) a^{\sum_i m_i} v.$$
Since $a \neq 0$, this implies that the submodule generated by $w$ contains $v$ and so is all of $\MM_\phi(a)$.   But  $w$ was
an arbitrary nonzero element, so $\MM_\phi(a)$ is irreducible in this case.

If $a = 0$,  then $\NN: = \bigoplus_{\bar k \neq \bar 0}  \mathbb C y(\bar k) v$ is a proper submodule.  \end{proof}

\subsection{$\phi$-imaginary Verma modules for $\G$} \label{phimage}

For $\phi: \mathbb{N} \rightarrow \{\pm \}$, next we construct a $\G$-module containing a submodule for
the Heisenberg subalgebra $\LL$  isomorphic to  $\MM_{\phi}(a)$.
Let $\lambda\in \H^*$ and assume $\lambda(c)=a$.
For $$
\SF_{\phi}= \D^{\mathsf{re}}_+
\cup \{n \delta\ |\  n \in \mathbb N, \phi(n)=+\}\cup \{-m \delta\
|\ m \in \mathbb N, \phi(m)=-\},
$$
where $\D^{\mathsf{re}}_+$ is as in \eqref{eq:posreal}, the spaces  $\G_{\SF_{\phi}} = \bigoplus_{\alpha \in \SF_{\phi}} \G_\alpha$ and
$\G_{-\SF_{\phi}} = \bigoplus_{\alpha\in  -\SF_{\phi}} \G_\alpha$ are subalgebras of $\G$
affording a triangular decomposition
 $\G = \G_{-\SF_{\phi}} \oplus \H \oplus
\G_{\SF_{\phi}}$ of $\G$.   Let $\Bo_{\phi}=\H\oplus \G_{\SF_{\phi}}$ be the Borel subalgebra corresponding
to $\SF_\phi$, and observe that $\Bo_{\phi} \supset  \mathbb C c \oplus \LL_{\phi}^+$.
Let $\mathbb C v_{\lambda}$ be a one-dimensional module for $\Bo_{\phi}$
with $\G_{\SF_{\phi}}v_{\lambda}=0$ and
$hv_{\lambda}=\lambda(h)v_{\lambda}$ for all $h\in \H$.

We say that the $\G$-module
$$\MM_\phi(\lambda): = \MM_{\SF_{\phi}}(\lambda) =\UU(\G)\otimes_{\UU(\Bo_{\phi})}\mathbb C v_{\lambda}$$
is a $\phi$-{\it imaginary Verma module}.    We identify $1 \otimes v_\lambda$ with $v_\l$.
The $\UU(\LL)$-submodule of $\MM_\phi(\lambda)$  generated
by $v_\l$   is isomorphic to $\MM_\phi(a)$.
 If  $\phi(n)=+$ for
all $n$,  then $\MM_{\phi}(\lambda)$ coincides with the imaginary
Verma module $\MM_{\SF}(\lambda)$ above with $\SF = \D_+$.

In the proposition below we collect some basic statements about
the structure of  $\MM_{\phi}(\lambda)$.   The proofs are similar to the
proofs of corresponding properties for the imaginary Verma modules in
\cite [Props. 3.4 and 5.3]{F2} and so are omitted.

\begin{proposition}\label{prop-phi-imaginary} Let $\l \in \H^*$ and assume
$\l(c)=a$.  If $a\neq 0$, then $\MM_{\phi}(\lambda)$ has the following properties.
\begin{itemize} \item  $\MM_{\phi}(\lambda)$ is a free
$\UU(\G_{-\SF_{\phi}})$-module of rank 1.
\item  $\MM_{\phi}(\lambda)$
has a unique maximal submodule and hence a unique irreducible
quotient.
\item  $\supp\left(\MM_{\phi}(\lambda)\right)=\bigcup_{\beta\in \dot
\QQ_+}\{\lambda-\beta+n\delta \mid n\in \mathbb Z\}$,  where $\dot \QQ_+$ is
the free abelian monoid generated by all the simple roots in the base $\Pi$ of $\dot \Delta_+$
(In the $\mathsf{A}_{2\ell}^{(2)}$-case, $\dot \QQ_+$ is the free abelian group generated by
the simple roots  $\alpha \in (\dot \Delta_s)_+$ and by the  $\frac{1}{2} \alpha$  for the simple roots $\alpha \in (\dot \Delta_l)_+$.)
\item If $\phi(k)\neq \phi(\ell)$ for
some $k,\ell \in \mathbb N$, then $\dim \MM_{\phi}(\lambda)_{\mu} = \infty$ for any
$\mu \in \supp\left(\MM_{\phi}(\lambda)\right)$.
\end{itemize}
\end{proposition}

We have the following irreducibility criterion for the modules
$\MM_{\phi}(\lambda)$.

\begin{theorem}\label{thm-phi-imaginary-irred}
Let $\l \in \H^*$, $\l(c)=a$.   Then $\MM_{\phi}(\lambda)$ is
irreducible if and only if $a\neq 0$.

\end{theorem}

This theorem will be proved in Section~\ref{sec-irred} as a
particular case of the main result (see Corollary~\ref{cor-Verma-irred}).

\section{$\mathbb Z$-graded modules for Heisenberg
algebras}\label{sec-Heis}

In this section we consider $\Z$-graded modules for the Heisenberg
subalgebra $\LL$ with nonzero action of the central element $c$.
To simplify the exposition,  we will assume we have an infinite-dimensional
Heisenberg Lie algebra $\HH= \mathbb C c \oplus \bigoplus_{i\in
\Z\setminus\{0\}} \mathbb C e_i$, where $[e_i, e_j]=\delta_{i,
-j}c$, and  $[e_j, c]=0$ for all $i \geq 1$ and all $j$.
The case of the Heisenberg subalgebra $\LL$ can easily be reduced to $\HH$
by choosing an orthogonal basis in each root space $\G_{k\delta}$
and a dual basis in $\G_{-k \delta}$ for each $k \geq 1$.

Let ${\mathcal{K}}$ denote the category of all $\Z$-graded
$\HH$-modules $\VV$ such that $\VV=\bigoplus_{j\in \Z}\VV_j$ and
$e_i\VV_j\subseteq \VV_{i+j}$.
The
irreducible modules in ${\mathcal{K}}$ on which $c$ acts as zero (which we say have zero level)
have been classified by Chari in \cite{C}.  Irreducible modules with
nonzero level,  but with $0<\dim_{\mathbb C} \VV_j<\infty$
for at least one $j$ have been described in \cite{F3}.

\begin{definition}\label{def-locally-finite} Let $\VV$ be a module for the Heisenberg Lie
algebra $\HH$.   Then we say
\begin{itemize}
\item[{(a)}]  $\VV$ has $i$-{\rm torsion} if $e_ie_{-i}$ has an
eigenvector in $\VV$; 
\item[{(b)}]  $\VV$ is a {\rm torsion}
module if $\VV$ has $i$-torsion for all $i\in \Z\setminus \{0\}$; \item[{(c)}]  $\VV$
is {\rm torsion free} if it has no $i$-{\it torsion} for any $i$;
\item[{(d)}]  $\VV$ is {\rm  locally finite} if for any $i\in \Z\setminus \{0\}$,
$e_i e_{-i}$ is locally finite on $\VV$, that is, \newline $\dim
\mathsf{span}_{\mathbb C}\{(e_ie_{-i})^kv  \mid  k\geq 0\}<\infty$ for any $v\in \VV$.
\item[{(e)}]  $\VV$ is {\rm  diagonal} if the $e_ie_{-i}$ have a common eigenvector in $\VV$ for all $i \in \Z \setminus \{0\}$.
\end{itemize}
\end{definition}

Clearly, diagonal and locally finite $\HH$-modules are
torsion modules.

\begin{lemma}\label{lem-gradation}
Let $V$ be an  $\HH$-module such that $c$ acts by the nonzero scalar $a$ on $\VV$, and
assume $w \in \VV$ is  an eigenvector for $e_j e_{-j}$ with eigenvalue $\lambda$
and for $e_i e_{-i}$ with eigenvalue $\mu$ for some $i \neq \pm j$.
Then  $e_j^r w$ and $e_{-j}^sw$ are eigenvectors for  $e_j e_{-j}$ with eigenvalues
$\lambda -r a$ and $\lambda+sa$ respectively, and they are eigenvectors for
$e_i e_{-i}$ corresponding to eigenvalue $\mu$.
\end{lemma}

\begin{proof}  The Heisenberg relations imply for $r,s \geq 1$ that
\begin{equation}\label{eq:eigenrels}
e_{-j} e_j ^r w = (\lambda - r a) e_j^{r-1} w \qquad \hbox{\rm and}  \qquad e_j e_{-j}^s w  = (\lambda + (s-1)a) e_{-j}^{s-1} w
\end{equation}
from which it follows $(e_j e_{-j}) e_j^r w = (\lambda-ra) e_j ^r w$ and $(e_j e_{-j}) e_{-j}^s w = (\lambda + s a)e_{-j}^s w$.
The remaining assertion is clear from the commuting properties in $\HH$.    \end{proof}

\begin{proposition}\label{prop-torsion}  Let $\VV$ be an irreducible
$\HH$-module with a scalar action of $c$.
\begin{itemize}
\item If $\VV$ has
$i$-torsion, then  $\VV$ has a countable basis which consists of
eigenvectors of $e_ie_{-i}$, that is,  $e_ie_{-i}$ is diagonalizable on
$\VV$.
\item If $\VV$ is a torsion module, then $\VV$ is locally finite.
\item If  $\VV$ is diagonal,  then the $e_ie_{-i}$ are
simultaneously diagonalizable on $\VV$ for all $i \in \Z\setminus \{0\}$.
\item If $\VV$ is a diagonal module, then $\VV$ is locally finite.
\end{itemize}
\end{proposition}

\begin{proof}
When $\VV$ is irreducible,  it is spanned by elements of the form 
$e_{j_1}e_{j_2}\ldots e_{j_k}v$,  where  $v \neq 0$ is an eigenvector of $e_{i}e_{-i}$.  The first
statement is obvious if the action of $c$ is zero,  and it follows
from Lemma~\ref{lem-gradation} if the action of
$c$ is nonzero.  For the second part, suppose $\VV$ is a torsion module for $\HH$.
Given $i$, choose a basis of $\VV$ consisting of eigenvectors of
$X_i: = e_ie_{-i}$.  Suppose $v=v_1+\ldots +v_s$,  where
$X_iv_j=\lambda_jv_j$ for $j=1, \ldots, s$. Then $\Pi_{j=1}^s
(X_i-\lambda_j)v=0$,  and hence $\VV$ is locally finite.   Assume next 
that  $\VV$ is a diagonal $\HH$-module,  and choose a common eigenvector $v$ of all the 
$X_i$.  Then $\VV$ has a spanning set consisting of
vectors of the form  $e_{j_1}e_{j_2}\ldots e_{j_k}v$, and applying Lemma~\ref{lem-gradation} as before,
we conclude that the $X_i$ are simultaneously diagonalizable on
$\VV$ for all $i$. The last statement also follows immediately.   \end{proof}

Next  we consider \emph{$\Z^\infty$-graded $\HH$-modules}  $\VV$.
 By that we mean $\VV =  \bigoplus_{\bar k \in \Z^\infty} \VV_{\bar k}$
where if $\VV_{\bar k} \neq 0$, then $\bar k = (k_1, k_2, \dots, )$, $k_i \in \Z$,  and $k_N = 0$ for all $N >> 0$.
Moreover,  we require that  $e_{\pm j}\VV_{\bar{k}}\subseteq
\VV_{\bar{k}\pm \zeta_j}$ holds for all $j\in \Z$ and $\bar{k}\in
\Z^{\infty}$. Here $\zeta_j\in \Z^{\infty}$ is the Kronecker
multi-index with $1$ at the $j$th place and $0$ elsewhere.

\begin{theorem}\label{thm-Z-infty-gradation}
Any $\Z$-graded irreducible diagonal $\HH$-module $\VV$  with nonzero level
has a $\Z^{\infty}$-gradation $\VV=\bigoplus_{\bar{k}\in
\Z^{\infty}}\VV_{\bar{k}}$.

\end{theorem}

\begin{proof}  Let $\VV = \bigoplus_{n = 0}^\infty \VV_n$ be a $\Z$-graded irreducible diagonal $\HH$-module
with $c$ acting by the scalar $a \neq 0$.   Then all $e_je_{-j}$ are simultaneously diagonalizable
on $\VV$ by Proposition~\ref{prop-torsion}.  In
particular, all the homogeneous spaces  $\VV_n$ have a basis consisting of common eigenvectors for the elements
$e_je_{-j}$, $j \in \Z\setminus\{0\}$.   For some $n$, choose  $0 \neq w\in
\VV_n$, a common eigenvector for all $e_je_{-j}$, $j\in
\Z\setminus\{0\}$.    Let $\HH(j)$ be the Heisenberg subalgebra generated by $e_j, e_{-j}$ for
each $j$.   Then by \eqref{eq:eigenrels},  $\UU(\HH(j)) w$ is spanned by the vectors $e_j^r w$, $e_{-j}^s w$ for $r,s \geq 0$.
Now for $\bar k = (k_1, k_2, \dots, k_m, 0,0,\dots) \in \Z^\infty$  consider
$y_1^{|k_1|} y_{2}^{|k_2|} \cdots y_m^{|k_m|}w$,
where $y_i = e_{i}$ if $k_i \geq 0$ and $y_i = e_{-i}$ if $k_i < 0$.   These vectors span $\VV = \UU(\HH)w$ by
the irreducibility of $\VV$.     For each such $\bar k \in \Z^\infty$, set
$\VV_{\bar k} =  \mathbb C y_1^{|k_1|} y_{2}^{|k_2|} \cdots y_m^{|k_m|}w$.
Then $ \VV = \bigoplus_{\bar k \in \Z^\infty} \VV_{\bar k}$.
The sum is direct,  since by
Lemma~\ref{lem-gradation},  the eigenvalues of the $e_{j}e_{-j}$ are sufficient to distinguish them.
Note that all components
$\VV_{\bar{k}}$ are at most one-dimensional and that $\VV$
remains
irreducible as $\Z^{\infty}$-graded module.   \end{proof}

\subsection{Irreducible diagonal modules over Heisenberg algebras}  \quad

Let $\HH$ be a Heisenberg algebra as in Section~\ref{sec-Heis}.
Denote by $\mathcal K_{\HH,a}$ the category of all  finitely-generated $\Z$-graded
 diagonal $\HH$-modules $\VV$,  where the central element $c$
of $\HH$ acts by the scalar $a\in \mathbb C$.  By $\mathcal
{ZK}_{\HH,a}$ we denote  the category of all  finitely-generated $\Z^{\infty}$-graded
$\HH$-modules $\VV$ with $c$ acting by $a$.
Theorem~\ref{thm-Z-infty-gradation} implies that any irreducible
module in $\mathcal K_{\HH,a}$ belongs to $\mathcal {ZK}_{\HH,a}$.
Therefore, to classify irreducible modules in $\mathcal K_{\HH,a}$
it is sufficient to classify irreducible  diagonal  modules in
$\mathcal {ZK}_{\HH,a}$.

Denote by $\mathcal{WZK}_{\HH,a}$ the full subcategory
of $\mathcal
{ZK}_{\HH,a}$ consisting of all finitely-generated
$\HH$-modules on which $e_ie_{-i}$ is diagonalizable for all $i \in \Z \setminus \{0\}$ (with a countable basis of eigenvectors).     Note that by
Proposition~\ref{prop-torsion} any irreducible object of
$\mathcal{K}_{\HH,a}$ belongs to
$\mathcal{WZK}_{\HH,a}$.

Let $\VV\in \mathcal{K}_{\HH,a}$ and $\VV=\bigoplus_{n\in
\Z}\VV_n$. We may assume that the generators of $\VV$ are
homogeneous (say in the spaces $\VV_{n_i}$ for $i=1,\dots, s$) and
are common eigenvectors for  $e_je_{-j}$, $j \in
\Z\setminus\{0\}$. The module $\VV$ is spanned by  the
 vectors $\prod_{j=1}^\infty  y_j^{|k_j|} w$, where $w \in \VV_{n_i}$ is a generator which is a common
 eigenvector for the $e_j e_{-j}$ for all $j$;   $y_j = e_j$  if $k_j \geq 0$ and $y_j = e_{-j}$
 if $k_j < 0$;  and $k_N = 0$ for $N >> 0$.   We assign to such a vector  $\prod_{j=1}^\infty  y_j^{|k_j|} w$ the
 gradation $n_i \zeta_1 + \bar k \in \Z^\infty$, where $\zeta_1$ has 1 in the
 first position and 0 elsewhere.  This makes $\VV$ into a $\Z^\infty$-graded module.
 Denote this
$\Z^{\infty}$-graded module by $\FF_1(\VV)$. Note that $\FF_1$ is
not well defined as a functor from $\mathcal{K}_{\HH,a}$ to
$\mathcal{WZK}_{\HH,a}$ since it depends on a choice of generators
in each $\VV$.

On the other hand, consider any $\MM\in \mathcal{WZK}_{\HH,a}$,
$\MM=\bigoplus_{\bar{k} \in \Z^{\infty}}\MM_{\bar{k}}$.  Now
define a $\Z$-grading on $\MM$ as follows: for any $n\in \Z$, 
set  $\MM_n=\bigoplus_{\bar{k}} \MM_{\bar{k}}$, where the sum is
over $\bar{k}=(k_1, k_2, \ldots) \in \Z^\infty$, such that
$\sum_{j=1}^\infty  k_j j=n$. Denote this $\Z$-graded module by
$\FF_2(\MM)$. Hence we obtain a functor
$$\FF_2: \mathcal{WZK}_{\HH,a}\rightarrow \mathcal{K}_{\HH,a}.$$
Note also that $\FF_2$ preserves irreducibility.

\subsection{Weight modules for Weyl algebras}
\quad

Fix a nonzero $a\in \mathbb C$.  For $n=1,2,\dots$,
consider  the $n$-th Weyl algebra $\AS_n$ with generators $x_i,\partial_i$, $i=1,\dots,n$,
and  defining relations  $[\partial_i, \partial_j] = 0  = [x_i,x_j]$ and $[\partial_i,x_j] = \delta_{i,j}a1$.
The algebra $\AS_n$ is isomorphic to the tensor
product of $n$ copies of  the first Weyl algebra $\AS_1$, which is the associative algebra
of  differential operators on the affine line.  We allow
$n$ to be $\infty$, in which case $\AS_\infty$ is just the
direct limit of the algebras $\AS_n$.

We will classify irreducible modules in $\mathcal K_{\HH,a}$ using
the classification of irreducible weight $\AS_{\infty}$-modules.
By specializing $c$ to $a$ and identifying $\partial_i$ with $e_i$
and $x_i$ with $e_{-i}$ for all $i>0$ (after a suitable
normalization), we obtain an isomorphism of the universal
enveloping algebra of $\HH$ modulo the ideal generated by $c-a$ with
$\AS_\infty$. Hence, any irreducible module $\VV\in  \mathcal{K}_{\HH,a}$
becomes a module for $\AS_\infty$.

Set $[n] = \{1,2, \dots, n\}$ (where $n = \infty$ is allowed and $[\infty] = \mathbb N$).
In  $\AS_n$, the elements $t_i=\partial_i x_i$, $i \in [n]$,
 generate the polynomial algebra $\DD=\C[t_i \mid i \in [n]]$, which is a
maximal commutative subalgebra of $\AS_n$. Denote by ${\mathcal G}$
 the group generated by the automorphisms $\sigma_i$, $i \in [n]$,  of $\DD$, where $\sigma_i(t_j) = t_j -
\delta_{i,j}a1$.  Then ${\mathcal G}$ acts on the set $\Max \DD$ of
maximal ideals of $\DD$.

 A module $\VV$ for $\AS_n$ is said to be a {\em weight module} if $\VV = \bigoplus_{\m \in
\Max \DD} \VV_\m$, where  $\VV_{\m} = \{v \in \VV \mid \m v = 0\}$
(see \cite{DGO}, \cite{BB}, and \cite{BBF}).

If $\VV$ is a  weight module and  $\VV_{{\m}}\neq 0$ for $\m\in \Max
\DD$, then ${\m}$ is said to be a {\em weight} of $\VV$, and the set  $ \{
{\m}\in \Max \DD \mid \VV_{\m} \ne 0\}$ of weights of $\VV$ is the {\it support\/} of $\VV$.
It follows easily from the fact that $x_i d = \sigma_i(d)x_i$ and $y_i d = \sigma_i^{-1}(d)y_i$ for all $i$ and all $d \in \DD$ that $x_i\,\VV_{\m}\subseteq \VV_{\sigma_i({\m})}$
and $\partial_i\,\VV_{\m}\subseteq \VV_{\sigma_i^{-1}({\m})}$.

Each weight module $\VV$ can be decomposed
 into a direct sum of $\AS_n$-submodules:

$$
\VV = \bigoplus_{\O} \VV_{\O}, \quad \quad \VV_{\O}: = \bigoplus_{{\m}
\in {\O}} \VV_{{\m}} \,
$$
where ${\O}$ runs over the orbits of ${\mathcal G}$ on $\Max \DD$.
 In particular,
 if $\VV$ is irreducible, then
its support belongs to a single orbit.

Following \cite{BBF},  we say that a maximal ideal $\m$ of $\DD$ is  a {\em
break with respect to $i \in [n]$} if $t_i\in \m$.  We let  $\II(\m)$
denote  the set of breaks of $\m$.   An orbit ${\O}$ is {\em degenerate with respect
to $i$} if $i \in \II(\m)$ for some $\m \in \O$.   Often we
simply say $\O$ is degenerate without  specifying $i$ or $\m$.
When $\II(\m) = \emptyset$ for all $\m \in \O$,  then
$\O$ is said to be nondegenerate.

A maximal ideal $\m$ of $\DD$ is a {\em maximal break with respect
to ${\II}  \subseteq [n]$} if
$t_i\in \m$ for each $i\in {\II}$,  and $t_j\not\in
\tau(\m)$ for each $j \in {\II^c}:=[n] \setminus
{\II}$ and each $\tau\in
{\mathcal G}$.   The {\it order} of the maximal break
is the cardinality of $\II$, which may be infinite.

\emph{We will always assume that a degenerate orbit $\O$  for
$\AS _n$ has a maximal break $\m$.}    The only time this assumption
is necessary is when $n = \infty$ (see \cite[Lem.~2.5]{BBF}). We say that an  
$\AS_n$-module is
{\em admissible} if its
 weights belong to either a nondegenerate orbit or  a degenerate orbit
$\O$  with a maximal break.

The classification of irreducible admissible weight $\AS_n$-modules was obtained in
\cite{BB},  and in \cite{BBF} for the case $n=\infty$. We briefly recall
this classification for the sake of completeness.

For a given orbit ${\O}$, we define the set ${\mathfrak B}_{\O}$
as follows.   If ${\O}$ is nondegenerate,  then any
maximal ideal gives a maximal break with respect
to the empty subset of $[n]$.  We fix a choice of
a maximal ideal $\m$ in $\O$,  and set  ${\mathfrak
B}_\O=\{\m\}$ and $\O_\m = \O$.    In this case
$\II(\m) = \emptyset$,  since there are no breaks.

Now let ${\O}$ be a  degenerate orbit,  and assume $\m$ is a fixed
maximal break in $\O$.  Let $\II = \II(\m)$ be  the set of breaks for $\m$.
Define
\begin{equation}{\mathfrak
B}_\O=\left \{ \Big(\textstyle{\prod_j}\, \sigma_{j}^{\delta_j}\Big)(\m)\, \Big | \, \delta_j \in\{0,1\}\ \hbox{\rm if}\ j \in  {\II}, \delta_j = 0 \  \hbox{\rm if} \  j\in {\II^c}, \,
 \ \hbox{\rm and only finitely many} \, \delta_{j} \neq 0\right\}.\end{equation}
The $\sigma_i$
commute, so the product is well-defined.  For each
$\p= \big( \prod_j \sigma_{j}^{\delta_j}\big)(\m) \in {\mathfrak B}_\O$,  we set

\begin{equation}
\O_\p:=\left \{\Big(\textstyle{\prod_j} \, \sigma_{j}^{\gamma_j}\Big)(\p) \, \Big | \,
\gamma_j=(-1)^{\delta_j+1}k, \ k\in\Z_{\geq 0} \ \text{if} \ j \in
{\II}, \  \hbox{\rm and} \ \gamma_j\in \Z \ \text{if} \ j \in
{\II^c} \right\},  \label{on} \end{equation}
where only finitely many $\gamma_j$ are nonzero.

Suppose first that ${\O}$ is a nondegenerate orbit of
${\mathcal G}$ on $\Max\,\DD$.   Set $\AS = \AS_n$ and

\begin{equation} \S(\O) =\bigoplus_{\n\in {\O}}\DD/\n,
\end{equation}

\noindent and define a left $\AS$-module structure on $\S({\O})$ by
specifying for $i \in [n]$ and $d \in \DD$  that

\begin{equation} x_i(d+\n):=\sigma_i(d)+\sigma_i(\n),
\qquad
\partial_i(d+\n):=t_i\sigma^{-1}_i(d)+\sigma^{-1}_i(\n).
\label{act} \end{equation}

\noindent As $\S(\O)$ is generated by $1 + \m$, we have that $\S(\O)
\cong \AS/\AS\m$ where $1+\m \mapsto 1+\AS\m$.

\medskip
Now assume that  ${\O}$ is degenerate, and $\m$ is the fixed
maximal break.  For $\p \in {\mathfrak B}_{\O}$ set

\begin{equation} \S({\O}, \p):=\bigoplus_{\n\in
{\O}_\p}\DD/\n,  \label{sop}
\end{equation}

\noindent where $\O_{\p}$ is as in (\ref{on}).   One can define a
structure of a left $\AS$-module on $\S(\O,\p)$ by the same formulae
as in (\ref{act}), but with the understanding that when the image is not in $\S(\O,\p)$, the
result is 0.  Assuming
$\p = \big(\prod_j \sigma_{j}^{\delta_j}\big)(\m)$,  we have in this
case $\S(\O,\p) \cong \AS/\AS \langle \p, z_i, i \in \II\rangle$ where
$z_i=x_i$ if $\p$ is a break with respect to $i$, and
$z_i=\partial_i$ otherwise.  The isomorphism is given by $1+\p
\mapsto 1 + \langle \p, z_i, i \in \II\rangle$.  It follows {f}rom
the construction that $\S({\O})$ and $\S({{\O}, \p})$ are
irreducible $\AS$-modules.

\begin{theorem}\label{firstiso}\cite[Thm.~4.7]{BBF}  Let $\O$
be an orbit of $\Max \DD$ under the group $\mathcal
G$. Then the modules $\S({\O})$ and $\S({{\O},
\p})$, where $\p\in {\mathfrak B}_{\O}$, constitute an exhaustive
list of pairwise nonisomorphic irreducible admissible weight
$\AS_n$-modules with support in $\O$.
\end{theorem}

Let ${\mathcal W}(\AS_n)$ be the category of all finitely
generated admissible weight $\AS_n$-modules, $n\leq \infty$, and
let $\HH_n$ denote the ($2n+1$)-dimensional Heisenberg Lie algebra
with basis $c, e_i,e_{-i}, i \in [n]$ (where $n = \infty$ is allowed, $[\infty] = \mathbb N$,
and $\HH_\infty = \HH$).  Then  we immediately have
the following.

\begin{corollary}\label{cor-category-equiv} \quad
Suppose $a\neq 0$.
\begin{itemize}
\item Every irreducible module of the category
${\mathcal{K}_{\HH_{n},a}}$ is isomorphic (up to an automorphism
of $\HH_n$) to an irreducible module in ${\mathcal W}(\AS_n)$  for
any $n<\infty$. \item Every irreducible module of the category
${\mathcal W}(\AS_n)$ is isomorphic (up to an automorphism of
$\HH_n$) to an irreducible module in ${\mathcal{K}_{\HH_n,a}}$ for
any $n$.
\end{itemize}
\end{corollary}

Note that ${\mathcal{K}_{\HH,a}}$ contains irreducible modules
which are non-admissible weight modules over $\AS_\infty$.   The first
such examples were constructed in \cite{MZ}.    Assume  $\VV = \bigoplus_{i \in \Z} \VV_i \in
{\mathcal{K}_{\HH,a}}$.   For any nonzero homogeneous $v\in \VV_i$,  let
$\mathsf{s}(v)=\{j\in \mathbb Z \setminus \{0\}  \mid e_jv=0\} \subset \Z$. Set $$\Omega=\{\mathsf{s}(v) \mid 0 \neq v\in
\VV_i, i\in \mathbb Z\}.$$   We say that $\VV$ is \,{\em admissible}\, if every
totally ordered subset of $\Omega$ (under containment of subsets)  has an upper bound.   Denote by
${\mathcal{AK}_{\HH,a}}$ (respectively $\mathcal{AZK}_{\HH,a}$)
the full subcategory of ${\mathcal{K}_{\HH,a}}$ (respectively
$\mathcal{WZK}_{\HH,a}$) consisting of admissible modules.  Let 
${\mathcal{AK}_{\HH_n,a}}$, $\mathcal{AZK}_{\HH_n,a}$, and $\mathcal{WZK}_{\HH_n,a}$
be defined similarly for $n \in \mathbb N$.  Then we
have

\begin{corollary}\label{cor-category-equiv-adm} \  Assume  $a\neq
0$.
Every irreducible module in  the category
${\mathcal{AK}_{\HH_{n},a}}$ is isomorphic (up to an automorphism
of $\HH_n$) to an irreducible module in ${\mathcal W}(\AS_n)$  for
any $n \geq 1$.


\end{corollary}

\begin{example} {\rm To illustrate Corollary~\ref{cor-category-equiv-adm},
 suppose $\O$ is nondegenerate,  and $\m$ is the designated maximal
ideal of $\O$. Consider the irreducible weight module for $\AS = \AS_\infty$ given by
$\S({\O})=\AS/\AS\m=\bigoplus_{\n\in \O}\DD/\n$, where $\DD=\C[t_i \mid i \in \mathbb N]$,  and $\DD/\n \cong \C$ for any $\n \in \O$. Then $\S({\O})$ becomes
a $\Z$-graded irreducible $\HH$-module if we set

$$\S({\O})=\sum_{k \in \Z} \S^{k}({\O}),$$
\noindent where

$$\S^{k}({\O})= \sum_{\eta_j \in \Z, \,  \sum_{j} j \eta_j =-k} \textstyle{ \DD\Big /\Big (\big(\prod_j  \sigma_j^{\eta_j}\big)(\m)\Big)},$$
where only finitely many $\eta_j$ are nonzero.    The homogeneous space
$\S^0(\O) = \DD/\m$. } \end{example}

As  a consequence of
 Corollary~\ref{cor-category-equiv-adm} we have
the following stronger version of
Theorem~\ref{thm-Z-infty-gradation}.

\begin{theorem}\label{theorem-irred-gr}
For nonzero $a \in \mathbb C$, there is  one-to-one correspondence
between the isomorphism classes of irreducible modules in the
categories $\mathcal{AK}_{\HH,a}$, $\mathcal{AZK}_{\HH,a}$ and
${\mathcal W}(\AS_\infty)$.
\end{theorem}

Combining Theorem~\ref{theorem-irred-gr} and
Theorem~\ref{firstiso},  we obtain  the classification
  of irreducible modules in $\mathcal{AK}_{\HH,a}$.

\subsection{Irreducible locally-finite modules over $\LL$}\label{sec-irred-loc-fin-L}

Now consider  the Heisenberg subalgebra $\LL = \mathbb C c \oplus
\bigoplus_{k \neq 0} \G_{k\delta}$ of the affine Lie algebra $\G$.
For each $k\in \Z$, $k\neq 0$, assume $\mathsf{d}_k = \dim \G_{k \delta}$
and  write $[\mathsf{d}_k] = \{1,\dots,\mathsf{d}_k\}$.   Choose a basis $\{x_{k,i}
\mid i  \in [\mathsf{d}_k]\}$ for $\G_{k\delta}$  so that  $[x_{k,i}, x_{-k,j}]=\delta_{i,j}kc$ for all
$i,j$. Then for every $i$ and $k$, the elements $x_{k,i}$ and $x_{-k,i}$ generate a Lie
subalgebra isomorphic $\HH_1$.

We will consider $\Z$-graded $\LL$-modules $\VV=\bigoplus_{j\in
\Z}\VV_j$ where $\G_{k \delta}\VV_j\subseteq \VV_{k+j}$ for all
$k$ and $j$.  One can easily extend
Definition~\ref{def-locally-finite} to the algebra $\LL$ substituting
basis elements $\{e_\ell\}$ by $\{x_{k,i}\}$.

Fix $a\in \C\setminus \{0\}$. Let  $\mathcal{K}_{\LL,a}$ be the
category of all $\Z$-graded diagonal $\LL$-modules $\VV$ where the
central element $c$ of $\LL$ acts by the scalar $a\in \mathbb C$.
Similarly one defines categories $\mathcal{AK}_{\LL,a}$ and
 $\mathcal{AZK}_{\LL,a}$. Clearly, all statements from the
 previous sections can be generalized to the setup of the
 Heisenberg algebra $\LL$.
In particular, there exists the Weyl algebra $\tilde{\AS}$,
generalizing $\AS_{\infty}$, which takes into account the dimensions
of the spaces $\G_{k\delta}$. Then Theorem~\ref{theorem-irred-gr}
has the following straightforward generalization for $\LL$.

\begin{corollary}\label{cor-irred-gr-G}  For $a \in \C\setminus \{0\}$,
there is  one-to-one correspondence between the isomorphism
classes of irreducible modules in the categories
$\mathcal{AK}_{\LL,a}$, $\mathcal{AZK}_{\LL,a}$ and ${\mathcal
W}(\tilde{\AS})$.

\end{corollary}

\begin{example}{\rm
Let $a\in \C\setminus\{0\}$ and $\phi:\mathbb N\rightarrow
\{\pm\}$ be any function. Then the $\phi$-Verma module
$\MM_{\phi}(a)$ is an irreducible object in the categories
$\mathcal{K}_{\LL,a}$ and $\mathcal{AZK}_{\LL,a}$. If $\phi(k)= +$
and $\phi(\ell)=-$ for some $k, \ell\in \mathbb N$,  then all the
homogeneous components of  $\MM_{\phi}(a)$ in the $\Z$-grading are
nonzero and infinite dimensional.}   \end{example}

\subsection{$\tilde{\phi}$-imaginary Verma modules for
$\G$}\label{sec-tilde-phi-imag}

  We will  generalize the construction of the
$\phi$-Verma modules for $\LL$ as follows.   Set
\begin{equation}\label{eq:Jset} {\tN} =\bigcup_{k\in \mathbb N}\{(k, i)\mid i\in [\mathsf{d}_k] \}.\end{equation}
Consider a function $\tilde{\phi}:{\tN} \rightarrow \{\pm\} $, and define
Lie subalgebras $\LL_{\tilde{\phi}}^{\pm}$ of $\LL$ by the following rule.
For  $k\in \mathbb Z \setminus \{0\}$,  we say that $x_{k,i} \in
\LL_{\tilde{\phi}}^{\pm}$  if  either  $k>0$ and $\tilde{\phi}(k,i)=\pm$,   or if
$k<0$ and $\tilde{\phi}(-k,i)=\mp$, where the $x_{k,i}$ are as in the first paragraph
of Section \ref{sec-irred-loc-fin-L}.
 Then $$\LL=\LL_{\tilde{\phi}}^{-}\oplus \mathbb C c\oplus
\LL_{\tilde{\phi}}^{+},$$  where  $\LL_{\tilde{\phi}}^{\pm}$
are abelian subalgebras of $\LL$.

\begin{remark}
The function  $\tilde{\phi}$ clearly depends on the initial choice of orthogonal  bases in the imaginary root spaces
$\G_{k\delta}$ with respect to the nondegenerate form on $\G$.  On the other hand, for each positive integer $k$, the number of $+$ and $-$ in the image of  $\tilde{\phi}$ does not depend on the choice of bases.

\end{remark}

Let $\mathbb C v$ be a one-dimensional representation of $\mathbb
C c\oplus \LL_{\tilde{\phi}}^{+}$ with $cv=a v$ for $a\in \mathbb
C$ and $\LL_{\tilde{\phi}}^{+}v=0$. Then we construct the
corresponding $\tilde{\phi}$-Verma module
$$\MM_{\tilde{\phi}}(a)=\UU(\LL)\otimes_{\UU(\mathbb C c\oplus \LL_{\tilde{\phi}}^{+})}\mathbb C
v.$$

If $\tilde{\phi}(n,i)=\tilde{\phi}(n,j)$  for all $n$ and all
$i,j\in [\mathsf{d}_n]$, then $\MM_{\tilde{\phi}}(a)$ is just a
$\phi$-Verma module for $\LL$, where $\phi(n)=\tilde{\phi}(n,i)$ for
each $n$ and any $i$. For any function $\tilde{\phi}$, the
$\tilde{\phi}$-Verma module $\MM_{\tilde{\phi}}(a)$  is an object in the categories
$\mathcal{K}_{\LL,a}$ and $\mathcal{WZK}_{\LL,a}$. One can easily see that
$\MM_{\tilde{\phi}}(a)$ is irreducible if and only if $a\neq 0$.

For any such function $\tilde{\phi}$ and any $\l\in \H^*$ with
$\l(c)=a$, one can construct the $\tilde{\phi}$-imaginary Verma
module $\MM_{\tilde{\phi}}(\l)$ over $\G$ generalizing the
construction of $\MM_{\phi}(\l)$ in the case of the function
$\phi:\mathbb N\rightarrow \{\pm\}$:

$$\MM_{\tilde{\phi}}(\lambda): = \MM_{\SF_{\tilde{\phi}}}(\lambda) =\UU(\G)\otimes_{\UU(\Bo_{\tilde{\phi}})}\mathbb C v_{\lambda}.$$

\begin{theorem}\label{thm-phi-tilde-imaginary-irred}
Let $\l \in \H^*$, $\l(c)=a$ and assume  $\tilde{\phi}:{\tN} \rightarrow
\{\pm\} $ is any function.   Then $\MM_{\tilde{\phi}}(\lambda)$ is
irreducible if and only if $a\neq 0$.

\end{theorem}

This theorem also will be proved in Section~\ref{sec-irred} as a
particular case of the main result (see Corollary~\ref{cor-Verma-irred}).

\subsection{Realization of locally-finite modules} \quad

Locally-finite $\LL$-modules have also been considered by Casati
\cite{Ca}. Following Casati's work we will construct realizations of
 irreducible locally-finite $\LL$-modules.

Let $\KK \subseteq  \mathbb N$ and set $\VV=\mathbb C[x_i, x_{-k} \mid i\in \mathbb N,\,
k \in \KK]$.   Then it is easy to verify  that the
following formulas define a representation of $\AS = \AS_\infty$  on $\VV$:

\begin{eqnarray*} && \partial_i\rightarrow \begin{cases} \displaystyle{\frac{\partial}{\partial x_i}} & \quad \hbox{\rm if} \ \ i \in \mathbb N\setminus \KK \\
 \displaystyle{\frac{\partial}{\partial x_i}+x_{-i}} & \quad \hbox{\rm if} \ \ i \in \KK, \end{cases} \\
&& x_i \rightarrow x_i   \hspace{1.02 truein} \hbox{\rm for all } \   i \in \mathbb N,  \\
&& c\rightarrow 1. \end{eqnarray*}

\noindent This module is isomorphic to the universal module $\VV_{\KK}$ over $\AS$ generated by a vacuum vector $v$, where
$\partial_i v=0$ for any $i\in \mathbb N\setminus \KK$.  Hence $\VV_{\KK}\simeq \AS/\mathsf{B_{\KK}}$, where
$\mathsf{B}_{\KK}$ is left ideal of $\AS$ generated by $\partial_i, \ i\in \mathbb N\setminus \KK$. Clearly, this module is not irreducible.

Now for each $k \in \KK$, fix  $\vartheta_k\in \mathbb C$,  and let $\vartheta = \{\vartheta_k \mid k \in \KK \}$. Then we can construct the following  quotient of
$\VV_{\KK}$. Let $\mathsf{B}_{\KK,\vartheta}$ denote the left ideal of $\AS$ generated by the elements  $\partial_i, i\in \mathbb N\setminus \KK,$  and
$x_k\partial_k-\vartheta_k$, $k\in \KK$,  and denote the quotient by $\VV_{{\KK},\vartheta}=\AS/\mathsf{B}_{\KK,\vartheta}$.

Suppose that $\vartheta_k \in \mathbb C \setminus \mathbb Z$ for all $k \in \KK$. In  this case,  $\partial_k$ and $x_i$ act injectively on  $\VV_{\KK,  {\vartheta}}$ for all $k \in \KK$ and all $i \in \mathbb N$.   Let $\AS(i)$ denote the rank one Weyl algebra generated by $\partial_i$ and $x_i$.  Using the irreducibility of the Verma module over $\AS(i)$ generated by a vacuum vector $v$ such that $\partial_i v=0$, we conclude that $\VV_{\KK, \vartheta}$ is an irreducible $\AS$-module.

This construction can be generalized to the
Heisenberg algebra $\LL$.  In doing this, we will adopt the
notation from Section~\ref{sec-irred-loc-fin-L}.
Then an irreducible $\LL$-module  $\VV$  is
    diagonal if
 the elements $x_{k,j}x_{-k,j}$ are simultaneously
 diagonalizable on $\VV$
for all $k \in \mathbb N$, $j  \in [\mathsf{d}_k]$.

Fix a nonzero $a\in
\mathbb C$, and take $\vartheta_{k,j}\in \mathbb C$ for $k\in
{\KK}$, $j\in [\mathsf{d}_k]$. Set $\vartheta = \{ \vartheta_{k,j} \mid k
\in {\KK},  j \in [\mathsf{d}_k] \}$.   Consider the $\LL$-module
$\VV_{{\KK},\vartheta,a}=\UU(\LL)/ \mathsf{B}_{{\KK},\vartheta,
a}$, where $\mathsf{B}_{{\KK},\vartheta, a}$ is the left ideal of
$\UU(\LL)$ generated by $x_{-k,j}$,  $k\in \mathbb N\setminus
{\KK}$,  $x_{k,j}x_{-k,j}-\vartheta_{k,j}$ for $k\in {\KK}$, $j\in
[\mathsf{d}_k]$ and $c-a$. Then $\VV_{{\KK}, \vartheta,a}$ is an
irreducible $\LL$-module if and only if for any $k\in \KK$ and any
$j\in [\mathsf{d}_k]$, $\vartheta_{k,j}$ is not an integer multiple of
$ka$. Moreover, applying Corollary~\ref{cor-irred-gr-G} we have

\begin{theorem}\label{thm-realiz-locally-finite}
Up to an automorphism of $\LL$, for any irreducible admissible
diagonal $\LL$-module $\MM$ of level $a\neq 0$ there exist a set
$\KK \subseteq \mathbb N$ and scalars $\vartheta_{k,j}$, $k\in
\KK$, $j\in [\mathsf{d}_k]$ with $\vartheta_{k,j}$ not an integer
multiple of $ka$ such that  $\MM$ is isomorphic to $\VV_{{\KK},
\vartheta,a}$.
\end{theorem}

\begin{corollary}\label{cor:ybasis}  Assume $\MM$ is an
irreducible $\LL$-module  isomorphic to $\VV_{{\KK}, \vartheta,a}$.   For $\KK \subseteq \mathbb N$, set $y_{k,j} = x_{-k,j}$ if $k \in \mathbb N\setminus \KK$ for all $j \in[\mathsf{d}_k]$, 
and let $y_{k,j} = x_{k,j}$ or $x_{-k,j}$ if $k \in \KK$
for all $j \in[\mathsf{d}_k]$.   Then for any nonzero
vector $v \in \MM$,
the vectors \begin{equation} \label{eq:basislocfin} \cdots y_{2,\mathsf{d}_2}^{p_{2,\mathsf{d}_2}} \cdots y_{2,1}^{p_{2,1}} y_{1,\mathsf{d}_1}^{p_{1,\mathsf{d}_1}} \cdots y_{1,1}^{p_{1,1}}v \end{equation}
with exponents $p_{k,j} \in \mathbb N \cup \{0\}$ for
all $k,j$ and only finitely many of them nonzero  form a
basis for $\MM$.
\end{corollary}

\section{Generalized loop modules}

In this section,  we study modules over the affine Lie algebra $\G$ induced from modules in
the category $\mathcal{K}_{\LL,a}$ via a construction analogous to
that of the loop modules in \cite{C}.
Let $\SF$ denote the set given by
$$
\SF= \SR \, \cup \,
\{n \delta \mid n\in \Z\setminus \{0\}\}.
$$
where $\SR = \D^{\mathsf{re}}_+$, the positive real roots  as in \eqref{eq:posreal}.
Set $\PP=\H\oplus \G_{\SF}$, where $\G_{\SF} = \bigoplus_{\beta \in \SF} \G_\beta$.
  Then $\PP=(\H+\LL)\oplus
\G_{\SR}$, where $\G_{\SR} = \bigoplus_{\beta \in \SR} \G_{\beta}$, and $\PP$  is a parabolic subalgebra of $\G$ with Levi
factor $\H+\LL$.

Let $\VV\in \mathcal{K}_{\LL,a}$, $\VV = \bigoplus_{k \in \Z} \VV_k$,   and assume $a \neq 0$.  Suppose
$\lambda\in \H^*$ is such that $\lambda(c)=a$. Extend the module
structure to $\PP$ by setting $\G_{\SR}\VV=0$,  and $hv=\lambda(h)v$
for any $v\in \VV_0$ and any $h\in \H$. Here $\VV_0$ is the
$0$-component of $\VV$ in the $\Z$-grading. The action
of $\H$ on the other components of $\VV$ in the $\Z$-grading differs only in the value of the
degree derivation; that is,  for any $w \in \VV_k$, $hw=(\lambda+ k\delta)(h)w$
for each $h\in \H$.  (Recall that $\delta$ is zero on $\h \oplus \C c$
and $\delta(d) = 1$.)

Now consider the induced $\G$-module given by
\begin{equation}\label{eq:genloop} \MM(\lambda, \VV)=\UU(\G)\otimes_{\UU(\PP)}\VV. \end{equation}

When  $\VV$ is an irreducible module in the category $\mathcal{K}_{\LL,a}$,
then $\MM(\lambda, \VV)$ is said to be  a {\it generalized loop module}.
When $\VV$ is a $\phi$-Verma module of $\LL$ for some function
$\phi:\mathbb N\rightarrow \{\pm\}$,  which has been extended to
a module for $\PP$ by setting $\G_{\SR}\VV=0$, then $\MM(\lambda, \VV)  \cong \MM_\phi(\lambda)$, the
$\phi$-imaginary Verma module for $\G$ as in Section \ref{phimage}.

\begin{proposition}\label{prop-gener-loop} Let $\l \in \H^*$ and suppose that
$\l(c)=a\neq 0$.  Let  $\VV$ be an irreducible module in
$\mathcal{K}_{\LL,a}$. Then
\begin{itemize} \item  $\MM(\lambda, \VV)$ is a free
$\UU(\G_{-\SR})$-module of rank 1.
\item  $\supp\big(\MM(\lambda,
\VV)\big)=\bigcup_{\beta\in \dot \QQ_+}\{\lambda-\beta+n\delta\mid n\in \mathbb
Z\}$ where $\dot \QQ_+$ is as in Proposition \ref{prop-phi-imaginary}.
\item  $\dim \MM(\lambda, \VV)_{\mu} =
\infty$ for any $\mu$ of the form $\mu=\lambda-\beta+k\delta$
for some  $\beta \neq 0$ and $k\in \Z$.
\item $\dim \MM(\lambda,
\VV)_{\mu}<\infty$ only if $\VV$ is a $\phi$-imaginary Verma module for some
$\phi:\mathbb N\rightarrow \{\pm\}$, $\phi(m)=\phi(n)$ for all
$m,n\in \mathbb N$ and $\mu=\lambda -(\phi(m)m)\delta$ for some
$m\in \mathbb N \cup \{0\}$.
\end{itemize}
\end{proposition}

\subsection{Irreducibility of generalized loop modules}\label{sec-irred}\quad
 
Let $\l \in \H^*$, $\l(c)=a\neq 0$,  and assume $\VV$ is a module in
${\mathcal{K}_{\LL,a}}$.      Set
$$\widehat{\MM}(\l, \VV)= \bigoplus_{k\in \Z}  \MM(\l, \VV)_{\l+k\delta}.$$      Then $\widehat{\MM}(\l,
\VV)$ is an $\LL$-submodule of $\MM(\lambda, \VV)$ isomorphic to $\VV$.

\begin{lemma}\label{lem-top-gener-loop}
Let $\l \in \H^*$, $\l(c)=a\neq 0$, and suppose that  $\VV$ is an irreducible module in
${\mathcal{K}_{\LL,a}}$. Then for any
nonzero submodule $\NN \subset \MM(\lambda, \VV)$ we have $\widehat{\NN}=\NN\cap
\widehat{\MM}(\l, \VV)\neq 0$.   \end{lemma}

\begin{proof} Let $\Pi = \{\alpha_1, \dots, \alpha_n\}$ be the
base of simple roots for $\dot \Delta$,   and assume $\alpha = \sum_{j=1}^n k_j \alpha_j$
where each $k_j$ is in $\mathbb Z_{\geq 0}$  (or in $\frac{1}{2}\mathbb Z_{\geq 0}$
for any $\alpha_j \in (\dot \Delta_l)_+$ in the $\mathsf{A}_{2\ell}^{(2)}$ case).
Set $\mathsf {ht}(\alpha) = \sum_{j=1}^n k_j$, the \emph{height} of $\alpha$.
The argument will follow the general lines of the proof of
\cite[Lem.~5.4]{F4} and will proceed by induction on the height.
Say $\alpha \leq \beta = \sum_{j=1}^n  \ell_j \alpha_j$ if $\mathsf{ht}(\alpha) < \mathsf{ht}(\beta)$
or if $\mathsf{ht}(\alpha) = \mathsf{ht}(\beta)$ and $k_1= \ell_1, \dots, k_s= \ell_s$, but $k_{s+1} < \ell_{s+1}$.

By Theorem \ref{thm-realiz-locally-finite}, we may assume $\VV \cong \VV_{{\KK}, \vartheta,a}$
for some $\KK \subseteq \mathbb N$ and  some $\vartheta$.
Let $\NN \neq 0$ be a submodule of $\MM(\lambda, \VV)$.   Assume
 $v$ is a homogeneous generator of $\VV$,
 and let $w \in \NN$ be a
nonzero homogeneous element.
Then \begin{equation}\label{eq:welt} w=\sum_{i\in \mathcal I} u_iv_i,\end{equation}
where we may suppose  that the $v_i$ are distinct monomial basis elements of
the form \eqref{eq:basislocfin}  and  the $u_i$ are linearly independent
homogeneous elements of
$\UU(\G_{-\SR})$.
We may suppose that for each $i \in \mathcal I$ there is some $\ell_i \in \mathbb Z$ such that  $u_i \in \UU(\G_{-\SR})_{-\beta+\ell_i\delta}$.

Initially assume  $\mathsf{ht}(\beta)\leq 1$,  so that $\beta$ is a  simple root
in $\dot \D_+$ (or is $\frac{1}{2}\alpha_j$ for some simple root $\alpha_j \in (\dot \D_l)_+$ in the
$\mathsf{A}_{2\ell}^{(2)}$ case).   Suppose $0 \neq x\in
\G_{\beta+m\delta}$ for some $m \in \mathbb Z$.   Then $xv_i=0$ for any $i$
and
\begin{equation}\label{eq:xw} xw=\sum_{i\in  \mathcal I}[x,u_i]v_i.\end{equation}   Here
$[x,u_i]\in \G_{(m+\ell_i)\delta}$ and $[x,u_i]\neq 0$ for all $i$ (which can be seen from the loop realization of $\G$).
Since the $u_i$ are linearly independent,  and $u_i \in \G_{-\beta+\ell_i \delta}$, which is
one-dimensional,  we have that $\ell_i\neq \ell_j$ if $i\neq j$.   Fix $i_{\bullet} \in \mathcal I$.
Now using the notation of Section~\ref{sec-irred-loc-fin-L},  we have that  $[x,u_{i_{\bullet} }]$ is
a linear combination of the basis elements $x_{m+\ell_{i_{\bullet} },j}$.
We may suppose that  $m$ was chosen with $|m|$ sufficiently
large so that $m+\ell_{i_{\bullet} }$ is not equal to $k$ for any $y_{k,j}$ occurring  in any of
the $v_i$,  and  so that at least one of the  $x_{m+\ell_{i_{_\bullet} },j}$ appearing  in $[x,u_{i_{\bullet} }]$
equals $y_{m+\ell_{i_{\bullet} },j}$ in Corollary \ref{cor:ybasis}  (any $x$-term  not equal to a corresponding
$y$-term will annihilate $v_{i_{\bullet}}$).  Then  $[x,u_{i_{\bullet}}]v_{i_{\bullet}} \neq 0$,
and we  have found a nonzero element $xw$  in $\widehat \NN$,  which gives the starting point for
induction on the height.

Suppose now that $\mathsf{ ht}(\beta)>1$.     A basis for $\UU(\G_{-\SR})$ consists of
monomials of the form $z_{\beta_t, n_t}^{p_t} \cdots z_{\beta_2,n_2}^{p_2} z_{\beta_1,n_1}^{p_1}$, where
$0 \neq z_{\beta_j,n_j} \in \mathfrak g_{-\beta_j+n_j \delta}$,   $\beta_1 \leq \beta_2 \leq \dots$, and
if $\beta_i = \beta_{i+1}$, then $n_i < n_{i+1}$.   Thus, we can assume  for $w$  in \eqref{eq:welt}
that

$$w = \sum_{i \in \mathcal I} u_i v_i = \sum_{i \in \mathcal I}  z_{\beta_{i,t(i)}, n_{i, t(i)}}^{p_{i,t(i)}} \cdots z_{\beta_{i,2}, n_{i,2}}^{p_{i,2}}z_{\beta_{i,1},n_{i,1}}^{p_{i,1}}v_i,$$
where for each $i$ we have $u_i = z_{\beta_{i,t(i)}, n_{i, t(i)}}^{p_{i,t(i)}} \cdots z_{\beta_{i,2}, n_{i,2}}^{p_{i,2}}z_{\beta_{i,1},n_{i,1}}^{p_{i,1}}$;  each factor  $z_{\beta_{i,j}, n_{i,j}}$ is basis element for $\mathfrak g_{-\beta_{i,j}+n_{i,j}\delta}$
and  these
basis elements are ordered as in the monomials  above;
$\sum_{j=1}^{t(i)} \beta_{i,j} = \beta$; and $\sum_{j=1}^{t(i)} n_{i,j} = \ell_i$.
We may suppose that we have indexed  the summands
so $\beta_{1,1} \leq \beta_{2,1} \leq \cdots $ and that among the $\beta_{i,j}$  with height equal to $\mathsf{ht}(\beta_{1,1})$,
 $p_{1,1}$ is minimal.

 Now suppose $x\in
\G_{\beta_{1,1}-m\delta}$ is nonzero for some $m \in \mathbb Z$, and observe that $xv_i= 0$ for each $i$ so that $xw$ is as in \eqref{eq:xw}.
Since  $[x,z_{\beta_{1,1}, n_{1,1}}] \in \mathfrak g_{(m+n_{1,1}) \delta}$, it  is
a linear combination of the $x_{m+n_{1,1}, j}$.  We assume that $m$ has been chosen with $|m|$ sufficiently large so that
$m+n_{1,1}$ is distinct from all the $k$ with $y_{k,j}$ occurring in some $v_i$, and so that
at least one of the $x_{m+n_{1,1}, j}$ equals
$y_{m+n_{1,1}, j}$.  Since  $[x, u_1]$ has  $z_{\beta_{1,t(1)}, n_{1, t(1)}}^{p_{1,t(1)}} \cdots z_{\beta_{1,2}, n_{1,2}}^{p_{1,2}}z_{\beta_{1,1},n_{1,1}}^{p_{1,1}-1}[x,z_{\beta_{1,1}, n_{1,1}}]v_i \neq 0$ appearing in it, we will have $0 \neq xw \in \NN$.
Because $\mathsf{ht}(\beta-\beta_{1,1}) < \mathsf{ht}(\beta)$,   we may apply induction to $xw$ to find a
 nonzero element of $\widehat \NN$.     \end{proof}

Lemma \ref{lem-top-gener-loop} immediately implies our main result
about the structure of generalized loop modules.

\begin{theorem}\label{theor-irred-gener-loop}
Let $\l \in \H^*$, $\l(c)=a\neq 0$,  and assume $\VV$ is an  irreducible module in
${\mathcal{K}_{\LL,a}}$.   Then $\MM(\lambda, \VV)$ is an irreducible $\G$-module.
\end{theorem}

As a consequence of this result, any irreducible module $\VV$ from the category
$\mathcal{K}_{\LL,a}$ with $a\neq 0$ and any $\l\in \H^*$ such
that $\l(c)=a$ will determine an irreducible module $\MM(\lambda,
\VV)$ for the affine Lie algebra $\G$. Let $\VV$ and $\WW$ be
irreducible modules from $\mathcal{K}_{\LL,a}$  with $a \neq 0$, and suppose
$\l, \mu\in \H^*$ with $\l(c)=\mu(c)=a$. Then the modules $\MM(\lambda,
\VV)$ and $\MM(\mu, \WW)$ are isomorphic if and only if $\VV$
and $\WW$ are isomorphic as  $\LL$-modules and $\l=\mu$ (up to a shift of gradation).

\begin{corollary}\label{cor-Verma-irred}
Let $\l\in \H^*$, $\l(c)\neq 0$, $\phi:\mathbb{N}  \rightarrow \{\pm\}$ any function, and let  $\tilde{\phi}:\JJ\rightarrow
\{\pm\} $ be as in Section~\ref{sec-tilde-phi-imag}. Then the $\phi$-imaginary Verma module
 $\MM_{\phi}(\l)$ and the
$\tilde{\phi}$-imaginary Verma module
 $\MM_{\tilde{\phi}}(\l)$ are irreducible.
\end{corollary}

\subsection{Partial generalized loop modules} \quad

Now we consider particular examples of generalized loop modules.
Assume $\II\subset \mathbb N$  and let $\phi:\mathbb N \setminus
\II \rightarrow \{\pm\}$ be  any function.   Set  $\KS_{\II}
=\mathbb C c \oplus \left(\bigoplus_{k\in \II} \G_{\pm
k\delta}\right)$ and let  $$\KS_\phi^\pm=\Bigg(\bigoplus_{n\in
\mathbb N\setminus {\II}, \, \phi(n)=\pm }
\G_{n\delta}\Bigg)\oplus \Bigg(\bigoplus_{n\in \mathbb N\setminus
{\II}, \, \phi(n)=\mp } \G_{-n\delta}\Bigg).$$ Then
$\KS:=\KS_{\II}\oplus \KS_\phi^+$ is a parabolic subalgebra of
$\LL$.      Let $\NN$ be an irreducible diagonal $\mathbb
Z$-graded module over the Heisenberg Lie algebra $\KS_{\II}$ with
nonzero level $a$.   Extend the action to a module structure over
$\KS$ by setting $\KS_\phi^+ \mathsf{N}=0$.   With these
ingredients, we construct  an induced diagonal $\LL$-module
$$\VV: = \VV_{\II, \phi}(\NN) =\UU(\LL)\otimes_{\UU(\KS)}\NN.$$    Then $\VV$ is the tensor product of the vector space $\NN$
with the Verma module over the Heisenberg Lie algebra $\KS_\phi := \KS_{\phi}^- \oplus \mathbb C c \oplus \KS_\phi^+$.
Standard arguments (compare Props. \ref{prop:phi-graded} and \ref{prop-irred-phi-Verma})  show

\begin{lemma} The $\LL$-module $\VV = \VV_{\II, \phi}(\NN)$
is $\mathbb Z$-graded and  irreducible when $a \neq 0$.
\end{lemma}

Let $\l\in \H^*$ be such that $\l(c)=a$, and suppose that $\MM(\lambda, \VV)$ is  the generalized
loop $\G$-module associated with $\lambda$ and $\VV = \VV_{\II, \phi}(\NN)$.      Alternately,
 we can construct an induced module directly from $\NN$ by
first making $\NN$ into a module for $\KS \oplus \G_{\SR}$  by having $\KS_\phi^+ \oplus \G_{\SR}$
be in the annihilator subalgebra of $\NN$.    Then we can induce to a module
$\UU(\G) \otimes_{\UU(\KS \oplus \G_{\SR})} \NN$  for $\G$.

\begin{corollary}  The $\G$-module $\UU(\G) \otimes_{\UU(\KS \oplus \G_{\SR})} \NN$
is isomorphic to  $\MM(\lambda, \VV)$ for $\VV = \VV_{\II,\phi}(\NN)$,
and hence it is irreducible.

\end{corollary}

\begin{proof}
It is sufficient to note that $\VV=\UU(\LL)\NN$ and to apply
Theorem~\ref{theor-irred-gener-loop}.
\end{proof}

\section{Acknowledgment}
 Part of
this work was done while V. Futorny and I. Kashuba visited the
University of Murcia. The support of the Seneca Foundation and the
hospitality of the University of Murcia are greatly appreciated.
This paper was discussed when G. Benkart, V. Futorny,  and I.
Kashuba  visited the Mathematical Sciences Research Institute
(MSRI) in Spring 2008.  They acknowledge with gratitude the
hospitality  and stimulating research environment of MSRI.    G.
Benkart is grateful to the Simons Foundation for the support she
received as  a Simons Visiting Professor at MSRI. V. Futorny  was
supported in part by the CNPq grant (301743/2007-0) and by the
Fapesp grant (2010/50347-9). The authors are grateful to the
referee for many useful suggestions which led to the improvement
of the paper.

\end{document}